\documentclass{article}

\usepackage{latexsym}
\usepackage{amssymb}
\usepackage{amsthm}
\usepackage{graphicx}
\newtheorem{theorem}{Theorem}
\newtheorem{lemma}[theorem]{Lemma}

\newcommand{\ra}{\rightarrow}
\begin{document}

\title{\bf Universal Cycles of Restricted Words}

\author{\Large KB Gardner, Anant Godbole \\East Tennessee State University \\ {\tt zbgg2@etsu.edu, godbolea@etsu.edu}}
\date{}

\maketitle

\begin{abstract}

A connected digraph in which the in-degree of any vertex equals its out-degree is Eulerian; this baseline result is used as the basis of existence proofs for universal cycles (also known as generalized deBruijn cycles or U-cycles) of several combinatorial objects. We extend the body of known results by presenting new results on the existence of universal cycles of monotone, ``augmented onto", and Lipschitz functions in addition to universal cycles of certain types of lattice paths and random walks.
\end{abstract}

\section{Introduction}A universal cycle, or U-Cycle, is a cyclic ordering of a set of objects ${\mathcal C}$, each represented
as a string of length $k$. The ordering requires that object $b = b_0b_1\ldots b_{k-1}$ follow object
$a = a_0a_1\ldots a_{k-1}$ only if $a_1a_2\ldots a_{k-1} = b_0b_1\ldots b_{k-2}$. U-cycles were originally introduced
in 1992 by Chung, Diaconis, and Graham \cite{cdg} as generalizations of de Bruijn cycles. As
an example, the cyclic string 112233 encodes each of the six multisets of size 2 from the
set [3]:=\{1, 2, 3\}. Another well-quoted example, from \cite{h}, is the string
$$1356725\;\;6823472\;\;3578147\;\;8245614\;\;5712361 \;\;2467836\;\;7134582\;\;4681258,$$
where each block is obtained from the previous one by addition of 5 modulo 8. This string
is an encoding of the $56 ={8\choose 3}$
3-subsets of the set $[8]=\{1, 2, 3, 4, 5, 6, 7, 8\}$. Chung, Diaconis and Graham \cite{cdg} studied U-Cycles of
\begin{itemize}
\item subsets of size $k$ of an $n$-element set (as in the above example);
\item set partitions; and
\item permutations (with a necessarily augmented ground set and the use of order isomorphism
representations, e.g., the string 124324 encodes each of the six permutations
of $[3] = \{1, 2, 3\}$ in an order isomorphic fashion, which is impossible
using the ground set [3]).
\end{itemize}
DeBruijn's Theorem states that U-Cycles of $k$-letter words on an $n$-letter alphabet exist for
all $k$ and $n$, as evidenced for $k=3, n=2$ by the cycle 11101000.  The proof of this theorem relies upon the fact that a connected digraph is Eulerian if and only if the in-degree of each vertex equals the out degree; in other words, the graph is balanced.  Many of the recent results on the existence of U-Cycles involve words as above, but with restrictions.  We next summarize some of these results, focusing on work that has been done at the East Tennessee State University REU, and in the ETSU course MATH 4010, over the last few years.  To set the stage, we recall one of the early papers in the area: Jackson \cite{j} had proved that 
\begin{theorem}
U-Cycles exist of all one-to-one functions from $[k]$ to $[n]$, i.e., of $k$-letter words over $[n]$ in which no letter repeats, provided that $k<n$ -- but not otherwise.
\end{theorem} An example with $k=2$ and $n=3$ is provided by $ABCBAC$.  The fact that U-Cycles do not exist when $k=n$ follows from the fact noted above about permutations, and the pigeonhole principle makes the problem meaningless for $k>n$.  Seeking to create an analog for onto functions, ETSU undergraduates Bechel and LaBounty-Lay \cite{ab} proved, generalizing the example 110100 for $k=3, n=2$, that
\begin{theorem}
U-cycles exist for all onto functions from $[k]$ to $[n]$, i.e., for $k$-letter words over $[n]$ that exhaust the alphabet, provided that $k>n$ but not otherwise.
\end{theorem}
\noindent They also proved, in a result that foreshadows Theorems 4 and 11 below as well as one of the main results, Theorem 21, of Section 2, that
\begin{theorem}
U-Cycles exist for words of length $k\equiv1\pmod2$ from $\{0,1\}$ with $\lceil k/2\rceil$ zeros and $\lfloor k/2\rfloor$ ones (or vice versa).  Moreover, U-Cycles of words of length $k\equiv0\pmod2$ from $\{0,1\}$ with an equal number of zeros and ones do not exist.
\end{theorem}
\noindent When $k=3$, Theorem 3 is exemplified by the cycle 011010.  An REU student, Arielle Leitner, proved the following four results in \cite{a}, including the following generalization of Theorem 3.
\begin{theorem}
U-Cycles of {\it equitable} $k$-letter words on $[n]$, i.e., words whose ``alphabet composition" is as evenly distributed as possible over $[n]$ exist iff $k\not\equiv0\pmod n$.
\end{theorem}
\begin{theorem}
U-Cycles of ``almost onto" functions (i.e., functions whose range excludes one point in the codomain) from $[n]$ to $[n]$ exist for $n\ge3$.  Likewise U-Cycles of non-bijections on $[n]$ exist.  (These are respectively U-Cycles of $n$-letter words on $[n]$ that either are missing one letter, or which do not contain all letters.)
\end{theorem}
Competition rankings are ones in which ties are possible, leading to some ranks being eliminated.  For example, 21245 is a ranking of five contestants $A,B,C,D,E$ in which $B$ wins; $A$ and $C$ are tied for second place; third place is taken up by the tie for second; and $D,E$ rank 4 and 5 respectively.  Such ordered rankings clearly represent words with restrictions.
\begin{theorem} 
U-cycles of competition rankings for $n$ players exist for each $n\ge2$.
\end{theorem}
The work of Horan and Hurlbert \cite{hh} is related to but distinct from Theorem 6.  

{\it Strong passwords} are length $k$ words on $[n]$ such that $[n]=\cup_{j=1}^qA_i; A_i\cap A_j=\emptyset$, and the word contains at least one element from each $A_i$.
\begin{theorem} U-Cycles of strong passwords as defined above exist as long as $k\ge 2q$.
\end{theorem}
Graphs on $k$ labeled vertices are, after all, words of length ${k\choose 2}$ on $\{0,1\}$ and should thus admit a U-Cycle by deBruijn's theorem.  However in \cite{bks}, 2008 REU students Brockman, Kay and Snively proved that
\begin{theorem}
There is a U-Cycle of all graphs on $k$ vertices that uses just the alphabet $[k]$ rather than $\left[{k\choose 2}\right]$.  
\end{theorem}
In addition, they proved that  
\begin{theorem}
On $k$ vertices, trees, graphs with $m$ edges, graphs with loops, graphs with multiple
edges (with up to $m$ duplications of each edge), directed graphs, hypergraphs, and
$r$-uniform hypergraphs all admit U-Cycles.
\end{theorem}
As noted by RET teachers Champlin and Tomlinson in \cite{champlin}, the fact that ${\cal A}$ admits a U-Cycle does not imply that ${\cal A}^C$ does.  However, they proved, in a series of results similar in spirit to Theorem 5, and which complemented Theorems 1, 2, 6, and 7 (and using the notation from those theorems) that \begin{theorem}\ 

(a) There exists a U-Cycle of non-injective functions if $k\ge4$.  

(b) There exists a U-Cycle of non-surjective functions if $k\ge n>2$.  

(c) There exists a U-Cycle of illegal rankings if $n\ne 1,3$.

(d) There exists a U-Cycle of non-strong passwords if $q\ge 3$.
\end{theorem}

The work of Antonio Blanca \cite{blanca} moved the agenda in a different direction.  First, he generalized Theorem  3 by proving that 

\begin{theorem}
U-Cycles of binary words of weight between $s$ and $t$, where $1\le s<t\le k$, exist; using the binary coding, these are the same as U-Cycles of subsets of size in the interval $[s,t]$.
\end{theorem}
Blanca also proved an extension of Theorem 11 for arbitrary alphabet sizes and words on that alphabet with weight in a specified range.  He proved results on Sperner families and chains of subsets, in work that was generalized in Theorem 13 below.  Most relevant to this paper, however, he proved a result on lattice paths of length $n$ that we will extend in Section 3.
\begin{theorem}  There exists a U-Cycle of lattice paths of length $n$ on the rectangular integer lattice in the plane, which begin at the origin, and consisting of steps of moves in the 4 standard directions $\{N, S, E, W\}$, and which end up at a distance $\le k$ from the origin.
\end{theorem} \noindent Note that this result { is} about words on a 4-letter alphabet in which the composition of the word is restricted by the equation
\[\vert\#N-\#S\vert+\vert\#E-\#W\vert\le k.\]

Graduate students Bill Kay and Andre Campbell \cite{campbell} proved results on words with weight in a range that improved the above-mentioned results of Blanca \cite{blanca}, but they also extended deBruijn's theorem in a different direction:
\begin{theorem}
U-cycles exist for the assignment
of elements of $[n]$ to the sets in any labeled subposet of
the Boolean lattice; de Bruijn's theorem corresponds to the case when
the subposet in question consists of a single ground element.
\end{theorem} Significantly, Theorem 13 was (though this is not obvious) a statement of U-Cycles of suitably restricted words.  The agenda in this paper was continued by King, Laubmeier, and  Orans \cite{king}, who proved, in another result in which the connection to words is not obvious:
\begin{theorem}
U-Cycles exist for all naturally labeled posets on $n$-elements.
\end{theorem}
\noindent In \cite{cdg}, it was shown that U-Cycles exist for the partitions of an $n$-element set into an arbitrary number of parts if $n\ge5$.  An example of such a U-Cycle, for $n=5$ is the following: $$DDDDDCHHHCCDDCCCHCHCSHHSDSSDSSHSDDCH$$
$$SSCHSHDHSCHSJCDC,$$ where, e.g. the 5-letter word $HCCDD$ represents the partition $1\vert23\vert45$.  Not much was known, however, about partitions of an $n$-element set into a fixed number of parts.  Improving and extending key work of ETSU undergraduates Elks and McInturff 
who focused on the case $k=2,3$, REU students Higgins, Kelley and Sieben \cite{higgins} proved:
\begin{theorem}  There exists an Eulerian cycle of all partitions of $[n]$ into $k$ parts; $2\le k<n$.
\end{theorem}  They then went on to prove both positive and negative results on when these Eulerian cycles could be ``lifted" to U-Cycles.

As the reader can see, much has been done but much has yet to be accomplished.  In this paper, we 
consider further restrictions on words based upon (i) the behavior of the associated discrete functions, and (ii) additional considerations for lattice paths.  These two sets of results are presented in Sections 2 and 3 respectively.  In Section 2 we will study new results in which analytic concepts such as monotonicity, letter composition, and rates of growth determine the set of allowable words.   As mentioned, Theorem 12 will be generalized in Section 3.  Some of the results in these two sections are easier to establish than others, and this is a hallmark of the theory.

\section{New U-Cycles of Discrete Functions}  In this section, we deal with U-Cycles of words that mirror standard growth criteria from Calculus (e.g., monotonicity, Lipschitzness, etc.)

\subsection{Monotone Non-Decreasing words.}

 For the 26-letter English alphabet, words that obey the standard lexicographic order of their constituent letters are called \textit{monotone non-decreasing}, with a similar definition holding for {\it monotone non-increasing} words.  For brevity we will only consider non-decreasing monotone words and refer to them simply as \textit{monotone}.  For example, if we use the customary order $a < b < c < \ldots < z$, then the word $aaabbccgglln$ is {monotone,} while the word $aaabbggcclln$ is not.  Since we wish to incorporate cyclic arrangements, we say more generally that a word on an alphabet $a_1< a_2<\ldots< a_n$ is {monotone} if it has a cyclic rearrangement $(\alpha_1,\ldots,\alpha_k)$ with $\alpha_i\le \alpha_{i+1}$ for each $i$.  For example, $gggkklabf$ is monotone via the cyclic rearrangement $abfgggkkl$, while $gggkklabl$ is not, because $l\nleq g$.
 
\begin{theorem}
U-Cycles of $k$-letter monotone words on an ordered $n$-letter alphabet exist for
all $k$ and $n$.
\end{theorem}

\begin{proof}
Consider the alphabet with lexicographic order: $a_1< a_2< \ldots< a_n$. We assume that $n\ge3$, since if $n=2$ all words are monotone and deBruijn's theorem applies.   
As is customary, we create a digraph $D$ with vertices being words of length $k-1$ that can be extended to a monotone word, and with there being an edge from one vertex to another if the last $k-2$ letters of the first coincide with the first $k-2$ letters of the second.  First, we show that this digraph is balanced.  For vertices of length $\geq 2$, there will be a first letter $a_i$ and a last letter $a_j$.  There are two possibilities:  either $a_i \leq a_j$ or $a_i > a_j$.  

Suppose that $a_i \leq a_j$.   Then the possible letters that can be appended to our word as a suffix are all those letters $a_s$ such that $a_j \leq a_s \leq a_n$ or $a_1 \leq a_s \leq a_i$,  i.e., letters in the range $(a_i,a_j)$ are disallowed.  
Therefore, the number of possible letters that can be appended as a suffix, i.e., out-degree,  is $i + (n - j) + 1$.  
Next, consider adding a prefix instead of a suffix.  As before, the possible letters that can be appended to our word as a prefix are all those letters $a_p$ such that $j \leq p \leq n$ or $1 \leq p \leq i$.   
Thus the in-degree is also $i + (n - j) + 1$, and 
it follows that $D$ is balanced.  

Suppose that $a_i > a_j$.  Then the possible letters that can be appended to our word as a suffix are all those letters $a_s$ such that $a_j \leq a_s \leq a_i$, so that the out-degree is $i-j+1$  
In a similar fashion, the in-degree can also be checked to be $i-j+1$, and 
thus $D$ is balanced.

Second, we show that the $D$ is connected.  It is sufficient to show that the graph is weakly connected by exhibiting a path from any vertex to the constant monotone word $a_1\ldots a_1$. As before, either $a_i \leq a_j$ or $a_i > a_j$.  

Suppose that $a_i \leq a_j$.
In this case, we know that letters $a_s$ such that $1 \leq s \leq i$ are allowed suffixes.  So we begin with our first word and add a suffix of the letter $a_1$.  The resulting word will end in $a_1$ and begin with a letter whose index is $p \geq 1$, and so we can append another $a_1$ as a suffix.  Continuing in this fashion, we get the target word composed entirely of the letter $a_1$, since appending the letter $a_1$ is always allowable.  
Suppose that $a_i > a_j$. In this case, we can append any letter $a_s$ with $a_j \leq a_s \leq a_i$, so we add one $a_j$, and then another, and so on until we have a word of length $k-1$ composed entirely of the letter $a_j$.  Now, since the index of the first and last letters are equal, we may append any letter we like as a suffix, so we append an $a_1$, and then another, and so on until  we have a word of length $k-1$ composed entirely of $a_1$'s. 
Thus, the digraph is  weakly connected and balanced, and so it is Eulerian, ergo a U-Cycle exists, by considering the concatenation of the edge labels in the Eulerian cycle.
\end{proof}

\subsection{Small constructions.}

We next offer some small constructions for U-Cycles on monotone words.  

Consider the binary alphabet.  It is trivial to show that a U-Cycle exists for words of length 2 and 3 on the binary alphabet (all words of these lengths are monotone).  Consider words of length four on the binary alphabet.  The words 0101 and 1010 are not monotone, but the rest are.
We begin with the edge 0000 and trace the Eulerian path as follows
$$0000 \rightarrow 0001 \rightarrow 0010\rightarrow 0100\rightarrow 1001\rightarrow 0011\rightarrow 0111$$

$$\rightarrow 1111\rightarrow 1110\rightarrow 1101\rightarrow 1011\rightarrow 0110\rightarrow 1100\rightarrow 1000,$$
so the resulting U-Cycle is 00010011110110.

Consider the alphabet $\{A, B, C\}$.  It is trivial to show that a U-Cycle exists for words of length 2 on this alphabet (all words of length 2 are monotone).  Consider words of length 3. The words ACB, CBA, and BAC are not monotone.  Once again we work with edge labels and construct the U-Cycle as follows:
$$AAA\rightarrow AAB\rightarrow ABA\rightarrow BAB\rightarrow ABB\rightarrow BBB\rightarrow BBC\rightarrow BCC$$
$$\rightarrow CCB\rightarrow CBC\rightarrow BCB\rightarrow CBB\rightarrow BBA\rightarrow BAA\rightarrow AAC\rightarrow ACA$$
$$\rightarrow CAC\rightarrow ACC\rightarrow CCC\rightarrow CCA\rightarrow CAB\rightarrow ABC\rightarrow BCA\rightarrow CAA, $$
so that the resulting U-Cycle is $AABABBBCCBCBBAACACCCABCA$.

\subsection{Lipschitz Words}

A Lipschitz function $f:{\mathbb R}\rightarrow{\mathbb R}$ is one for which $|f(y) - f(x)| \le C|y-x|$ for all $x$ and $y$, where $C$ is a constant independent of $x$ and $y$.  In other words, the values of $f$ for successive integers $x$ and $x+1$ differ by at most $C$.  Now, consider analogously a $k$-letter word $(\alpha_1,\ldots,\alpha_k)$ on the \textit{ordered and cyclic alphabet} $a_1<a_2<\ldots<a_n$ in which each successive letter is within $c$ letters of the preceding one.  We call such strings \textit{Lipschitz words}.

\begin{theorem}
A U-Cycle of $k$-letter Lipschitz Words on an $n$-letter alphabet exists for all $k$ and $n$.
\end{theorem}

\begin{proof}
First, we show that the resulting digraph is balanced.  Consider a Lipschitz word of length $k-1$.  We construct digraph edges by appending allowable suffixes and prefixes to the word.  By definition, the value of each allowable suffix may vary by at most $c$ from the last letter in our word.  Since we can move $c$ letters in either direction, and since we can also repeat the same letter, the number of possible suffixes for a vertex, and therefore the out-degree on our digraph, is $2c+1$.  Similarly, allowed prefixes are under the same restriction, so the number of possible prefixes for a Lipschitz word, and therefore the in-degree on our digraph, is also $2c+1$, proving balancedness.

To show connectedness, consider a Lipschitz word of length $k-1$ which begins with the letter $a_i$ and ends with the letter $a_j$.  
Note that we can move from this word to another word that begins with $a_{j \pm r}$ such that $r \leq c$ by simply appending $a_{j \pm r}, r \leq c$ as a suffix.  Assume we want to arrive at a target word whose first letter is $a_{j \pm m}$, such that  $m >c$.  Since $m >c$, we can write $m$ as a sum: $c + \ldots+c+ d = m$, where $d \leq c$.  So, to construct the required word, we first add a letter $a_{j \pm c}$ in the appropriate `direction' (addition for increasing value, subtraction for decreasing), continuing until the difference between the last letter of penultimate word and the first letter of the target word is $d \leq c$, which we can reach in a single step. We then construct the rest of the word by repeating the above process.  

Since the digraph is connected and balanced, it is Eulerian, and thus a U-Cycle exists as claimed.
\end{proof}

\subsection {Cyclically Appearing Word Categories}

Several of the papers mentioned in Section 1 contain results on U-cycles for words which must contain at least one letter from a set of categories.  In this subsection, we ask that these letters alternate, thus providing U-Cycles for an analog of functions whose graphs follow a  ``zig-zag" pattern.  In \cite{champlin}, Champlin and Tomlinson had proved \begin{theorem}There exists a U-Cycle of  alternating $k$-letter words on
a $n=n_v+n_c$-letter alphabet that consists of $n_v$ ``vowels" and $n_c$ ``consonants"
if either $k$ is even, or if $k$ is odd and $n_v = n_c$. \end{theorem}  

\noindent If we think of vowels and consonants as categories, we may also consider words constructed of letters separated into $c$ ordered disjoint categories, such that each subsequent letter ``cycles" to the next category.  That is, if the first letter belongs to $C_1$, the next letter belongs to $C_2$, and so on to $C_c$, and then the order cycles again.

\begin{theorem}
A U-Cycle of $k$-letter words on an $n$-letter alphabet that cycles through each of $c$ different disjoint categories of letters $C_1,C_2,\ldots,C_c$ exists if  $k = ac+2$ for $a \in \mathbb{N}$.
\end{theorem}

\begin{proof}
  As always, if edge words are of length $ac+2$, then vertex words will be of length $ac+1$.  

Consider whether an edge word of length $c+1$ is possible (in this case, we think of $a$ as being 1.)  It will be associated with vertex words of length $c$.  Without loss of generality, such a word will be of the form $l_{1,C_1}l_{2,C_2}...l_{c,C_c}$, where $l_{j,C_j} \in C_j$.   Since $|C_1|$ is not necessarily equal to $|C_c|$, we cannot claim a balanced digraph for words of length $c$.  It follows that words of length $ac$, simply some multiple of $c$, will also not necessarily be balanced.

Now, consider an edge word of length $c+2$, associated with a vertex words of length $c+1$.  Without loss of generality, such a word will be of the form $l_{1,C_1}l_{2,C_2}...l_{c,C_c}l_{c+1,C_{1}}$.  In particular, since $l_{1,C_1}, l_{c+1,C_{1}} \in C_1$, $i(v)= o(v)$, and thus the graph is balanced.  It follows that the same will hold for any multiple of $c$, hence the graph is balanced for words of length $ac+2$.

If two letters $l_i, l_j \in C_{\ell}$, we say that they have a {\em shared condition}.  Similarly, $l_i \in C_{m}, l_j \in C_{\ell}$ will be said to have {\em different conditions}.

In order to construct a new word on our digraph from an existing one, we have shown that the first and last letters of our vertex words must have a shared condition, and the new vertex word we are constructing must also begin and end with letters of shared condition.  Without loss of generality, consider the first letter of a vertex word.  There are two possibilities: either the first letter of our target word has a shared condition with the first letter of our starting word, or it does not.  

If it does not, then the first letter of the target word can be appended to the end of our vertex word as normal after cycling the word around as needed, and the target vertex word constructed one letter at a time.

If the first letter of our new vertex word does have a shared condition with the first letter of our current word, simply append placeholder letters with different condition until the correct first letter can be added. From here we construct the new vertex word as before.  Thus, the digraph is connected.
\end{proof}
\subsection{Application:  Random Walks on the Honeycomb Lattice}

One might wonder in what ``practical" situation one has equal numbers of vowels and consonants that must alternate as in Theorem 18.  In this subsection, we provide a concrete illustration, namely walks on the honeycomb lattice.  The notion of ``lattice paths," studied in Section  3, is difficult to pin down on the honeycomb lattice  due to the lack of a well-defined Cartesian coordinate system, but it is interesting to consider the set of random walks on the honeycomb lattice pictured in Figure 1.

\begin{figure}[h!]
\centering
\includegraphics[height=2.5in]{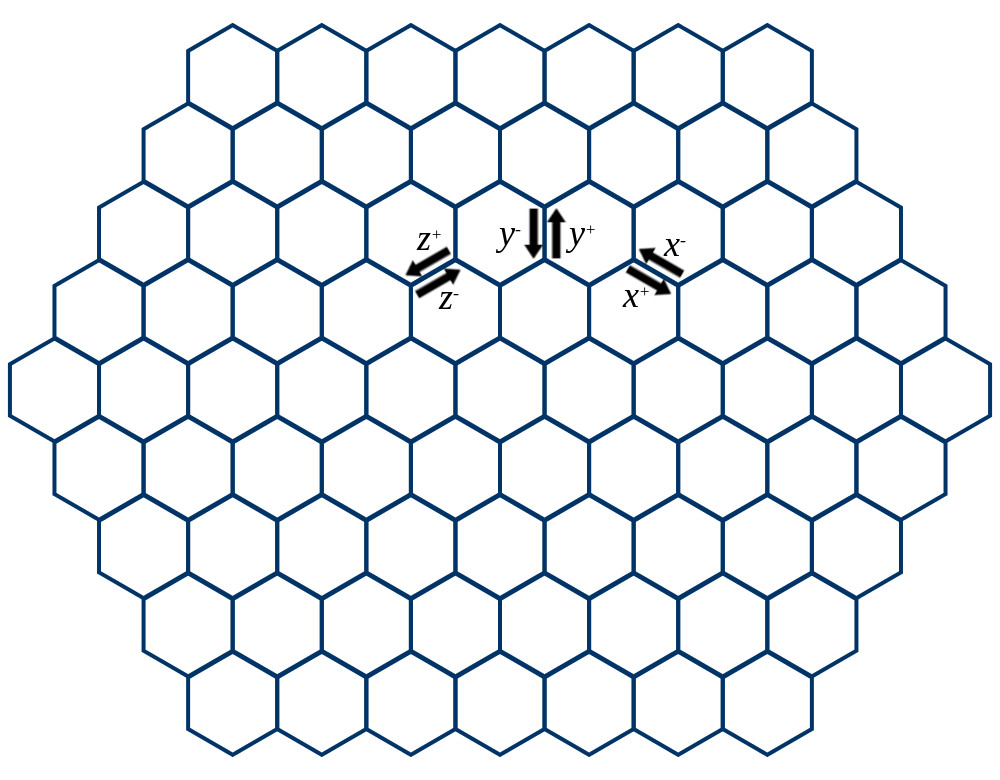}
\caption{Honeycomb Lattice.}
\label{fig:fig5}
\end{figure}

It is useful to define the directions a walk can take.  A honeycomb lattice is similar to a three-dimensional space, but each ``step" taken on the honeycomb limits the directions in which one can ``step" next.  We define the three directions in terms of the three ``axes" of the lattice: $x,y,$ and $z.$  Then, each of these is coupled with a direction to fill the alphabet we will use to compose honeycomb lattice walks.  Thus, our alphabet is defined as $\{x^+, x^-, y^+, y^-, z^+, z^-\}$.
Since we are not using cartesian coordinates, it is not necessary to define an ``origin."  The origin will simply be the arbitrary vertex where the walk begins.
Since each ``step" in the walk must be followed by a specific array of next ``steps,"  we list the possible next steps in Table 1.  In particular, note that the signs must alternate along our walk, and so the analogy with alternating vowels and consonants is complete when we consider the three letters $x^+, y^+$, and $z^+$ to be the vowels and the others to be the consonants.

\begin{table}
\centering
\begin{tabular}{c|c}
    Last & Next \\
    $x^+$ & $\{x^-, y^-, z^-\}$\\
    $x^-$ & $\{x^+, y^+, z^+\}$\\
    $y^+$ & $\{x^-, y^-, z^-\}$\\
    $y^-$ & $\{x^+, y^+, z^+\}$\\
    $z^+$ & $\{x^-, y^-, z^-\}$\\
    $z^-$ & $\{x^+, y^+, z^+\}$\\
\end{tabular}
\caption {Permitted steps.}
\end{table}

\begin{theorem}
A universal cycle of random walks on the honeycomb lattice of length $n$ exists for all $n$.
\end{theorem}

\begin{proof}

A simple application of Theorem 18; it would be interesting to see which other more complicated graphs, or lattices in $d$-dimensions, lend themselves to an analysis that would culminate in the conclusion that random walks on that graph/lattice are U-cyclable via Theorems 18 or 19.  \end{proof}

\subsection{Augmented Onto Words}

Generalizing the notion of onto words, consider a word in which every letter must appear at least once, but no letter may appear more than twice.  More generally, a $k$-letter word on an $n$-letter alphabet in which each letter must appear at least $a$ times and at most $b$ times (for $a,b\geq 1, a< b$) is called an $a-b$ \textit{augmented onto} word. In terms of functions we are looking at those for which $a\le \vert f^{-1}(\{j\})\vert\le b$ for each $j=1,2,\ldots,n$.  Note that for $b=a+1$, this is the same notion as that of ``equitable words" from Theorem 4, and, in fact, we focus on the case of $a=1, b=2$ in the next result.  The proof we offer, however, allows for generalization to the case of arbitrary $a,b$, and is quite different from the one in \cite{a}.

\begin{theorem}
A U-Cycle of 1-2 augmented words exists for all $n, k$ such that $n+1 \leq k \leq 2n-1$.  
\end{theorem}

\begin{proof}
Since our word is onto, we know from \cite{ab} that we must have $k \geq n+1$ for a U-Cycle to exist.  Furthermore, if $k=2n$, the situation would correspond to ``doubly onto" words (that is, those in which every letter appears exactly twice), and thus will not give a U-Cycle.  This means that the condition $k+1 \leq n \leq 2k-1$ is necessary for a U-Cycle of 1-2 augmented words to exist.





As always, the vertices of the digraph will be labeled with $k-1$ letter words.  There are two possibilities for the vertices: each of these $k-1$ letter words will be either onto, or nearly onto; that is, each word will contain every letter in the alphabet, or it will be missing exactly one letter.  In either case, vertices are words in which each letter occurs at most twice.  If a vertex $v$ is missing exactly one letter of the alphabet, it will have $i(v) = o(v) = 1$, because each incident edge must add the missing letter so that our $k$-letter edge word is onto.   If a vertex $v$ is onto, it will have a number of singleton letters $r$, and a number of paired letters $2(n-r)$ (that is to say $n-r$ pairs).  When a letter is added to generate a $k$-letter word,   notice we cannot add any of the letters already paired and still get a ``legal" word on our edge; we must add a singleton, which will create a new pair.  Thus our edge words will all have $r-1$ singletons and $n-r+1$ pairs.  Since the number of allowed edges corresponds to the number of singleton letters in our vertex word, $i(v) = o(v) = r$ when the word is onto.  Therefore the graph is balanced.  Note that we must have $r+2(n-r)=2n-r=k-1$, which implies that we must have  $r=2n-k+1$ to begin with.

Notice that when our edge words are at maximum length, $k=2n-1$, this will create vertex words with the fewest possible singleton letters with which to construct new words.  We claim that this is the most limiting case, so if we can show connectedness in this case, it will also be shown in less limited cases.
Let $k=2n-1$.  The vertices in such a digraph will contain words of length $k-1 = 2n-2$.  Notice that the vertex words will always have an even number of letters.  Thus when such a word is nearly onto, it will be missing one letter, and contain two of every other letter.  Also, if such a word is onto, it will contain exactly two singletons, and the remaining letters will be in pairs.  Therefore, for any $k-1$ letter word, $i(v) = o(v) \in \{1, 2\}$.

Notice that in our digraph, each subsequent edge ``adds" a letter to the end of our word, and ``drops" a letter from the beginning.  For this reason, we can discuss a means of choosing edges by which we can build our target word in terms of merely ``adding" a letter and ``dropping" a letter, with the knowledge that this corresponds to taking steps to traverse our digraph.  To show weak connectedness, we connect our starting vertex word to a fixed target vertex word, namely $A_1 A_2 ... A_{k-2} A_{k-1}=a_1a_1\ldots a_{n-1}a_{n-1},$ where the alphabet is $\{a_1,\ldots,a_{n}\}$.  We will construct our target one letter at a time, by manipulating the existing word until we can add the $A_i^{th}$ letter in its proper order using the following process.  Start adding letters to the existing word until there are no letters $a_1$ and then add two $a_1$s in succession.  Then cause the letter composition to have at most one $a_2$, and maintain this situation until the block $a_1a_1$ is once again at the end of the word (we may have to start rebuilding the target word anew for this to happen), and then add $a_2$.  Drop the other $a_2$ as soon as possible and maintain the word with just one $a_2$ until the block $a_1a_1a_2$ reappears at the end of the word, possibly by starting from scratch, when we can add the second $a_2$ causing the new ``successful block" to be $a_1a_1a_2a_2$.  Then add one and then another $a_3$ at the end of the word as before, and continue until we reach the target word.  This strategy works since we are always guaranteed two singletons when there is at least one singleton.  This process is illustrated for $n=5$ with starting word $bdabdece$ and target word $aabbccdd$.  We proceed as follows:
\[bdabdece\rightarrow dabdecec\rightarrow abdececb\rightarrow bdececb{a}\rightarrow dececb{aa}\rightarrow ececb{aab}\rightarrow\]
\[cecb{aab}d\rightarrow ecb{aab}dd\rightarrow cb{aab}dde\rightarrow b{aab}ddee\rightarrow { aab}ddeec\rightarrow {\bf abddeecc}\]
\[\ra bddeecca\ra ddeeccaa\ra deeccaab\ra eeccaabb\ra eccaabbd\ra ccaabbdd\]
\[\ra caabbdde\ra aabbddec\ra {\bf abbddece}\ra bbddecea\ra bddeceaa\ra ddeceaab\]
\[ \ra deceaabb\ra eceaabbc\ra ceaabbcd\ra eaabbcdd\ra aabbcdde\ra {\bf abbcddee}\]
\[\ra bbcddeea\ra bcddeeaa\ra cddeeaab\ra ddeeaabb\ra deeaabbc\ra eeaabbcc\]
\[\ra eaabbccd\ra aabbccdd\]
In the above example, the words in boldface represent when we have to start rebuilding the partial target word when it appears as the prefix to the entire word, but only a one step detour is needed to accomplish this.
The above example is entirely canonical and the process works for all initial words and all alphabet sizes.

Consider the case when $k=2n-2$.  Vertex words will have length $k-1=2n-3$, and we will have two cases:  When vertex words are nearly onto, they will contain one missing letter, one singleton and the remaining letters in pairs, and thus $i(v) = o(v) = 1$. 
When vertex words are onto, they will contain exactly 3 singletons, and the remaining letters in pairs.  Such words have $i(v) = o(v) = 3$.  To exhibit weak connectivity, we adjust the  target vertex to $a_1a_2a_2\ldots a_{n-2}a_{n-2}a_{n-1}a_{n-1}$.  As $k$ decreases, the number of singletons in almost onto vertices increases, affording greater flexibility in the same algorithm.  In the most extreme case, we have  $k=n+1$, and we let the target word be $a_1a_2\ldots a_{n-2}a_{n-1}a_{n-1}$.  To prove weak connectivity, we could introduce successive letters in the target word one at a time, introducing double occurrences as a means of introducing subsequent letters.  For example with $n=5$ and $k=6$, we illustrate a path from $cebad$ to $abcdd$ as follows (with boldface entries indicating the addition of another letter in the target word):
\[cebad\ra ebadc\ra badce\ra adceb\ra dceba\ra {\bf cebab}\ra ebabd\ra babdc\ra \]
\[abdce\ra bdcea\ra dceab\ra {\bf ceabc}\ra eabcd\ra abcdd.\]
Variations of the same algorithm may be checked to establish weak connectivity for all vertex word lengths in the allowable range, and thus
%
we conclude that a U-Cycle of $k$-letter 1-2 augmented words on an $n$-letter alphabet exists for all $n$ and $k$ with $k+1 \leq n \leq 2k-1$.
\end{proof}

We now consider the general case of augmented onto words.

\begin{theorem}
A U-Cycle of $k$-letter \textit{a-b augmented} words onto an $n$-letter alphabet exists for all $n, k$ such that $an+1 \leq k \leq bn-1$.
\end{theorem}
\begin{proof} The restrictions on the word length are clearly necessary as before.
As in Theorem 20, the vertices of the digraph will be labeled with $k-1$ letter words.  There are two possibilities: each of these $k-1$ letter words either contain a single letter that appears $a-1$ times, or else all letters will appear between $a$ and $b$ times.  In the former case we have $i(v) = o(v) = 1$, and in the latter case we have $i(v)=o(v)=m$, where $m$ is the number of letters that appear between $a$ and $b-1$ times (we call these letters {\em non-maximal}).

Consider the case when $k=bn-b+1$.  Since vertex words are $b$ short of being entirely maximal, then we are guaranteed to have at least two non-maximal letters whenever all letters are represented, and so we will be able to apply the algorithm from Theorem 21 to get to a fixed target vertex such as $a_1a_1\ldots a_1\ldots a_{n-1}\ldots a_{n-1}$ consisting of $b$ repeats of each letter.  So, the graph is connected when $k=bn-b+1$.

Consider the case when $k=bn-1$.  Then edge words will be of length $bn-2$ - that is, their length will be two letters short of every letter being maximal.  However, these two letters are \textit{not necessarily distinct,} meaning that we are not guaranteed to always have an additional letter as we were in the algorithm for Theorem 21.  In other words, while it is easy to say how many letters we \textit{cannot} add (it is $m$, the number of maximal letters), it is hard to say how many distinct letters are {not maximal}.  There are two possibilities.  Either all letters but one will be maximal, and one letter have order $b-2$ in our word, or all but two letters will be maximal, and those two letters will have order $b-1$.  

\textit{Case 1: One letter of order $b-2$.}  Since there is only one non-maximal letter, we must add it.  This will ``cycle out" the first letter in our word.  If the first letter in our word is the same as our non-maximal letter, then it is still the only non-maximal letter (since when we added it once, and cycled it out once, the order did not change), and we must add it again.    Since the only time we are forced to add a letter is if there is one letter of order $b-2$, we will never be forced to make a letter maximal.  Eventually, we will raise that letter's order to $b-1$, and thus force Case 2.

\textit{Case 2: Two letters of order b-1.} In this case, we may add either letter we wish.
Since we are never forced to make a letter maximal, this shows that the algorithm for Theorem 21  holds when $k=bn-1$, and so the graph is connected.  Other values of $k$ lead to easier analyses as with Theorem 21.
\end{proof}

\section{U-Cycles on Lattice Paths}
Lattice paths are a well-studied set of combinatorial objects; see, e.g., \cite{Mo}.  In two dimensions, a lattice path of length $n$ is a sequence of points $P_1, ..., P_n$ of $\mathbb{Z}\times \mathbb{Z}$
such that the $\ell_1$ distance between $P_i$ and $P_{i+1}$ is 1 for $i = 1, ..., n-1$. It is common to use a string over the alphabet ${N, S, E, W}$ to describe these paths, where $N,S$ correspond to positive and negative movement in the $y$ direction, respectively, and $E,W$ correspond to positive and negative movement in the $x$ direction, respectively.

Suppose that we are given two positive integers $n$ and $k$, and let $P_{n,k}$ be the set of the words over our alphabet corresponding to all lattice paths of length $n$ that begin at the origin and end up at a distance of at most $k$ from it.  We will refer to elements of $P_{n,k}$ both as paths and as words.  For instance, in the case of $n = k = 3$, $P_{3,3}$ would be the set of all paths ending at any destination $(x, y)$ with the restriction that $|x| + |y| \leq 3$. Of course, there are several ways of constructing any a path to an endpoint. For example we can get to (2, 1) via the paths $EEN, ENE, NEE$, which all have length 3.

It has been shown in \cite{blanca} that a U-Cycle of lattice paths in $P_{n,k}$ exists in the two-dimensional cartesian plane for all $n,k$.  We will extend this result in this section to three dimensions.

\subsection{Three Dimensional Lattices}

Let $P_{n,k}$ be the set of the words on the alphabet $\Sigma=\{N, S, E, W, U, D\}$ corresponding to three-dimensional lattice paths of length $n$ that start at the origin and end up at a distance of at most $k \geq 1$ from it.  Let $\Pi_{k}$ be the set of all points in the three-dimensional lattice with $|x|+|y|+|z|\le k$.

We first state and prove the following two key auxiliary results:

\begin{lemma}
In the digraph defined below, every vertex has degree 6 or it is connected to some vertex of degree 6.
\end{lemma}

\begin{proof}
 The digraph $D$ in question consists of all strings over $\Sigma$ representing three dimensional lattice paths from the origin, of length $n-1$, and ending up at a distance $\le k+1$ from the origin.  There is an edge between $v=a_1a_2\ldots a_{n-1}$ and $w=b_1b_2\ldots b_{n-1}$ if $a_j=b_{j-1}$ for $2\le j\le n-1$ and the concatenation of $v$ and $w$ represents a lattice path in $P_{n,k}$.  Consider a vertex word in $D$.  If our vertex word's endpoint $v$ lies within the interior of $\Pi_k$, i.e., in $\Pi_{k-1}$, it must have $i(v) = o(v) = 6$.  Other vertices have the same indegree and outdegree which might equal 1, 2, or 3.  All such vertices lie on the boundary of $\Pi_k$ or $\Pi_{k+1}$ and the differences in degrees depend on whether they lie on the corner, edge, or face of the polyhedron.

We first show that every vertex of degree 1 is connected to a vertex of degree 6.
Every vertex of degree 1 corresponds to a path whose endpoint is in a corner of $\Pi_k$ or $\Pi_{k+1}$.

Consider a degree 1 vertex $v_1$ associated with a word whose endpoint lies one move outside of $\Pi_k$.  Such a vertex word must lead to an edge-word whose endpoint lies in the corner of $\Pi_k$.  However, the associated vertex it connects to might still be in $\Pi_{k+1}$ since it might contain the same numbers of the sybols in $\Sigma$.  For example, with $n=16$ and $k=4$, the vertex $$SSSSUDEWNNNNNNNNN$$ points only towards $$SSSUDEWNNNNNNNNNS$$ which also has degree $1$.  Assume without loss of generality that the vertex $v_1$ has $k+1$ more $U$ steps than $D$ steps; an equal number of $E$ and $W$ steps; and an equal number of $N$ and $S$ steps.  In order to connect it, in multiple steps, to a vertex of degree six, we cycle it until the letter $U$ appears at the front of the associated word, drop it, and replace it with a $D$.  This makes $\#(U)-\#(D)=k-1$, while not changing the numbers of other letters, so that the endpoint of the vertex is in $\Pi_{k-1}$ and has degree six.  The same strategy works for vertices on the corner of $\Pi_k$.

\begin{figure}[h!]
\centering
\includegraphics[height=2.5in]{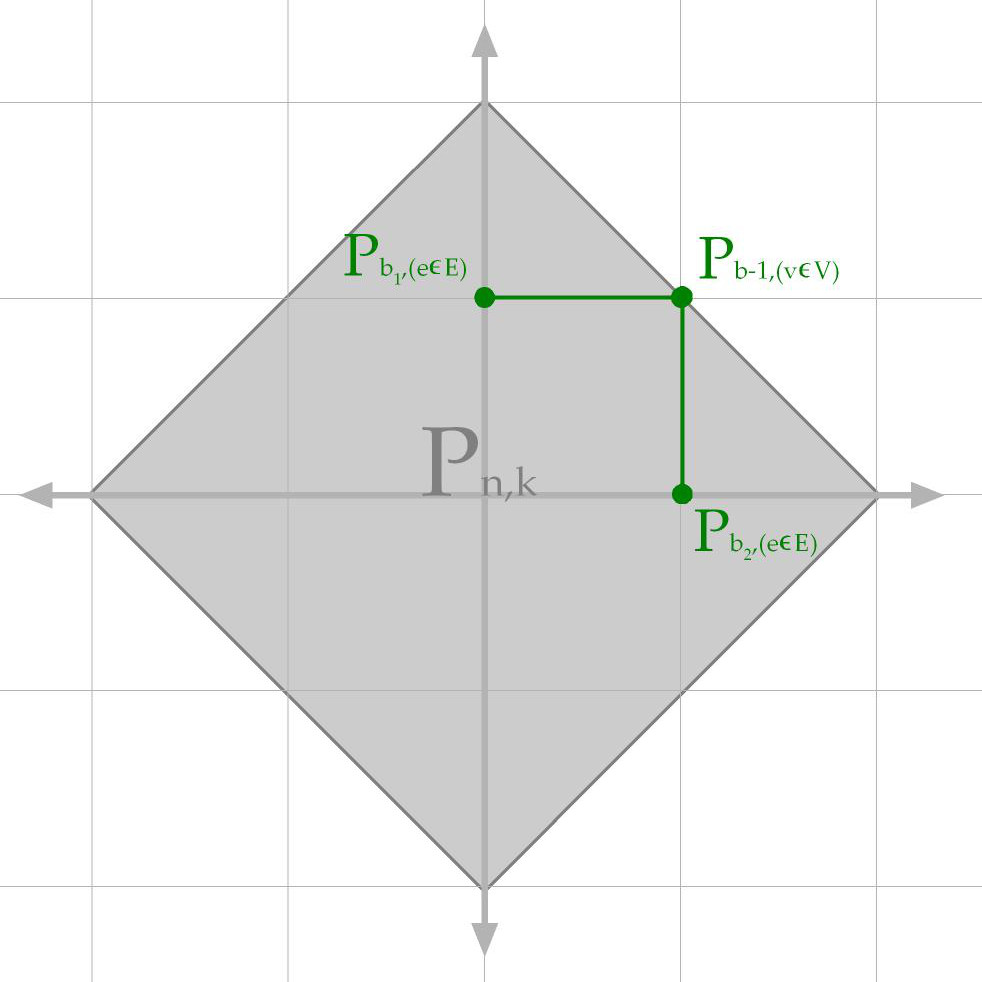}
\caption{Implications of degree 2.  Cross-sectional view of $P_{n,k}$.}
\label{fig:fig1}
\end{figure}

Let $v_2$ be a vertex of degree 2 on an edge of $\Pi_{k+1}$.  Such a vertex might point at two others both of which are also on the boundary of $\Pi_{k+1}$ as seen by the example where $$SSUUDDNNNNEEEEEWW$$ transitions to \[SUUDDNNNNEEEEEWWS\] or \[SUUDDNNNNEEEEEWWW.\]  The general strategy for this case is the following:  Assume without loss of generality, that such a vertex has an equal number of $U$s and $D$s, and that we have $\#(N)-\#(S)\ge2$ or $\#(E)-\#(W)\ge 2$.  (If both of these differences are 1, then we must have $k=1$ which is not allowable.)  Assuming that $\#(N)-\#(S)\ge2$, we cycle the letters of the vertex until an $N$ appears at the front of the associated word, drop it, and replace it with an $S$.  This makes $\#(N)-\#(S)$ two smaller than before, while not changing the numbers of other letters, so that the vertex is in $\Pi_{k-1}$ and has degree six.  The same strategy works for vertices on an edge of $\Pi_k$.

\begin{figure}[h!]
\centering
\includegraphics[height=2.5in]{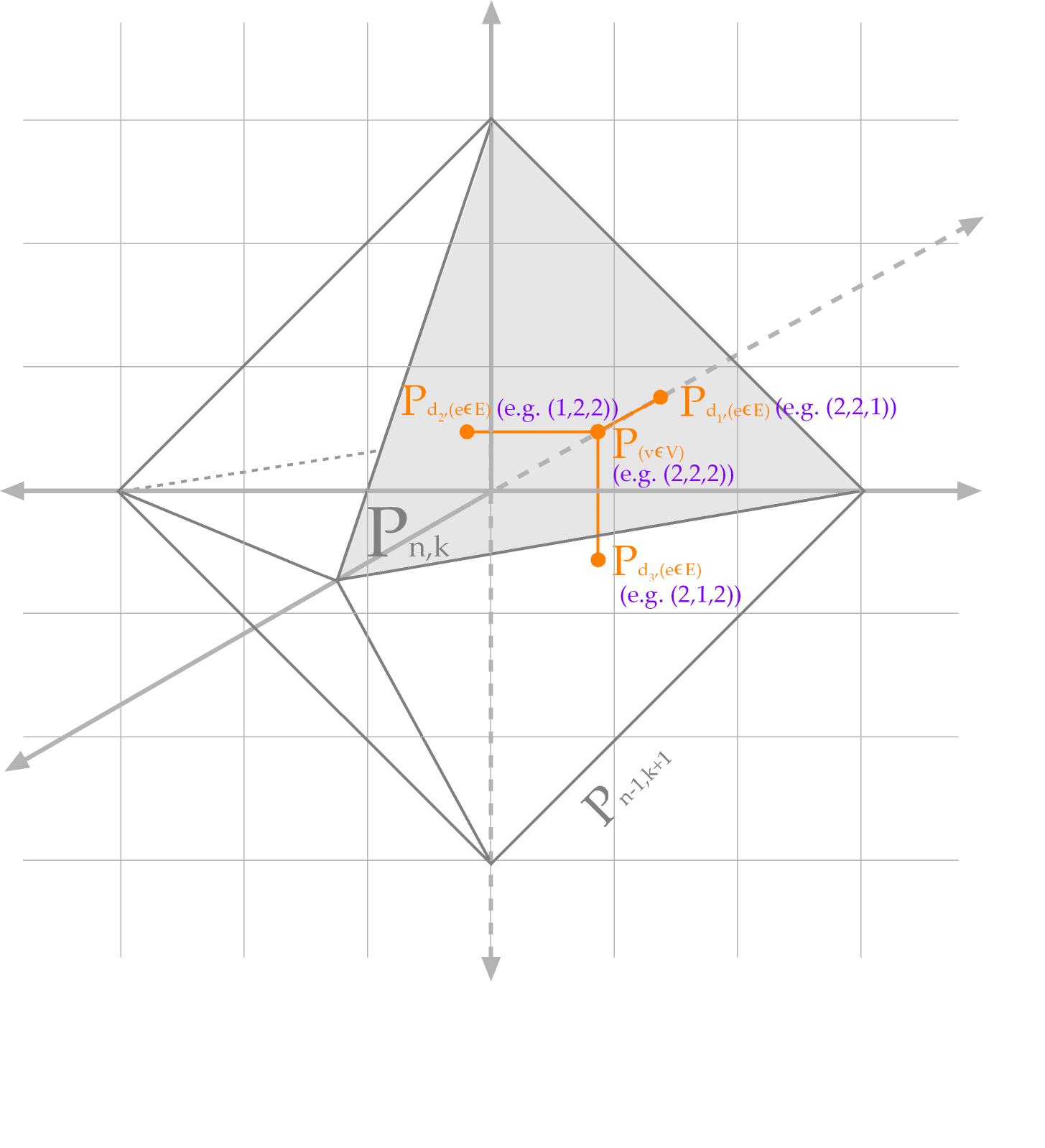}
\caption{Implications of degree 3.}
\label{fig:fig2}
\end{figure}

Finally let $v_3$ be a vertex of degree 3 on a face of $\Pi_{k+1}$.  The strategy for this case is the following:  Assume without loss of generality, that such a vertex has $\#(U)-\#(D)\ge2$ or $\#(N)-\#(S)\ge2$ or $\#(E)-\#(W)\ge 2$.  (If each of these differences are 1, then we must have $k=2$ which is not allowable.)  Assuming that $\#(N)-\#(S)\ge2$, we cycle the letters of the vertex until an $N$ appears at the front of the associated word, drop it, and replace it with an $S$.  This makes $\#(N)-\#(S)$ two smaller than before, while not changing the numbers of other letters, so that the vertex is in $\Pi_{k-1}$ and has degree six.  The same strategy works for vertices on a face of $\Pi_k$.

This proves the result.
\end{proof}

\begin{lemma}
If $n$ is odd, every vertex in $V$ is connected to a vertex whose corresponding path ends at the origin, and if $n$ is even, every vertex in $V$ is connected to a vertex corresponding to a path ending at distance 1 from the origin.
\end{lemma}

\begin{proof}
For a vertex to have a path that ends in the origin, it must have equal numbers of each of the letters in a component direction, e.g., an equal number of $N$s and $S$s.  This is only possible if the vertex length, $n-1$ is even.  If, however, $n$ is even, then a vertex can only correspond to a path which is ``one move away" from ending at the origin, or in other words, paths ending in $(0,0,1)$ or $(0,1,0)$, or $(1,0,0)$.

Since every vertex in $V$ has degree 6 or it is connected to some vertex of degree 6, we need only consider beginning vertices with degree 6.  Pick one, and call this vertex $v_i$.

Let $n_i, s_i, e_i, w_i, u_i$, and $d_i$ be the number of $N, S, E, W, U,$ and $D$ moves in the path of length $n-1$ associated with $v_i$.  This vertex will be associated with a ``word" containing $n-1$ letters in the amounts specified.  By definition, $v_i$ corresponds to a path that ends at a distance at most $k+1$ from the origin.  Such a path must also end within one of 8 simplex-shaped regions surrounding the origin.  Without loss of generality, let us assume that all three are positive, so we will be appending $W, S$, and $D$ moves to our word to connect it to a path ending at the target vertex  near or at the origin.  

Since $v_i$ has degree six, any swap of adjacent elements is allowable, by rotating until the two letters $ab$ in question are at the beginning of the word and then replacing the $a$ by a $b$ and then the $b$ by an $a$.  Using swaps of adjacent elements, we re-order the path associated with $v_i$ such that all the $U$ moves are at the end, and preceded in turn by all the $D$, $N$, $S$, $E$ and $W$ moves.  Since the path associated with $v_i$ ends at a point where $x,y,z$ are positive, it lies somewhere to the North, East, and Up from the origin.  Exploiting the degree six ``trump card", we rotate appropriately and replace as many $U$s by $D$s as needed to bring the frequency of these letters to within one, and then do the same with the other two pairs of letters.  Now if $n$ is odd, we see from parity considerations that either each of these differences is zero, or exactly two of these differences are one (and the last zero).  In the first case, we are at a vertex whose path begins at the origin and we are done.  In the second case, we replace the letter with the larger frequecy in one category by a letter with the smaller frequency in the second category to arrive at a vertex whose path begins at the origin.  If $n$ is even we see that the number of unit differences is either one or three.  We are done in the former case since we are at a vertex whose path begins at one of the three points $(1,0,0,), (0,1,0),$ or $(0,0,1)$.  In the latter case we reach one of these three vertices via a single swap of elements.  This completes the proof.
\end{proof}

\begin{theorem}
For $n\ge k+1\ge 4$, there exists a universal cycle for all paths in $P_{n,k}$ in three dimensions. (If $n\le k$, all lattice paths are valid and the result also holds true by deBruijn's Theorem.)
\end{theorem}

\begin{proof}
 Let ${\mathcal V}$ be the set of all words on the alphabet ${N, S, E, W, U, D}$ corresponding to lattice paths of length $n - 1$ which start at the origin and end at a distance $d \leq k + 1$ from it.  These words are associated with vertices in our digraph.  
As before, define the digraph $D =< {\mathcal V}, {\mathcal E} >$, with the edgeset ${\mathcal E}$ being defined by letting there be an edge from $v=a_1a_2\ldots a_{n-1}$ to $w=b_1b_2\ldots b_{n-1}$ if $a_j=b_{j-1}$ for $2\le j\le n-1$ and the concatenation of $v$ and $w$ represents a lattice path in $P_{n,k}$.    
We seek to exhibit an Eulerian path in $D$, by showing that $D$ is balanced and weakly connected.

To show balancedness, we proceed as in the proof of Lemma 23.  In fact, the fact that the digraph is balanced is a consequence of the fact that it is merely the numbers of letters of each type that determine whether or not a vertex is in $P_{n,k}$, and the entire situation is illustrated in Figure 4.
%

\begin{figure}[h!]
\centering
\includegraphics[height=2.5in]{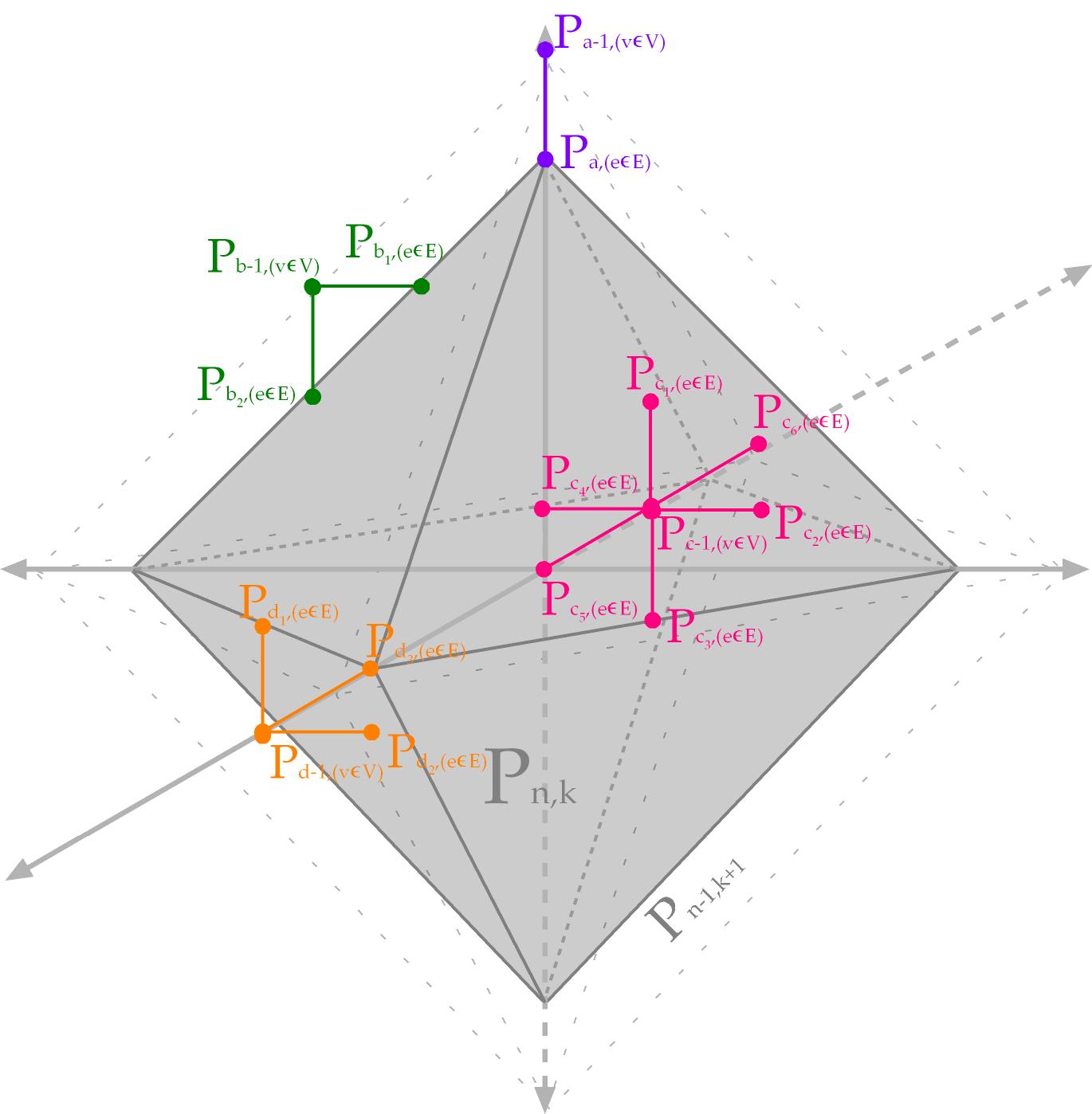}
\caption{Three-dimensional endpoints.}
\label{fig:fig4}
\end{figure}


Finally, we show weak connectedness.  By Lemma 24, if $n$ is odd, every vertex in $V$ is connected to a vertex whose corresponding path ends at $(0, 0, 0)$, and if $n$ is even, every vertex in $V$ is connected to a vertex corresponding to a path ending at (without loss of generality) $(0, 0, 1)$.   We can now show (as in the proof of Lemma 24) that each such terminal vertex is connected to a canonical one, say one with letters appearing in consecutive blocks of the same letter. Therefore, the digraph is weakly connected, and thus Eulerian, and so a U-Cycle exists.
\end{proof}
\noindent{\bf Remark.}  It is not too hard to formulate and prove a result similar to Theorem 25 in the case of $m$-dimensional lattice paths.








 


 

\end{document}